\documentclass[11pt]{article}
\usepackage{amsthm}
\usepackage{amsfonts}
\usepackage{amsmath}
\usepackage{amssymb}
\usepackage{colonequals}
\usepackage{epsfig}
\usepackage{url}
\newtheorem{problem}{Problem}
\newtheorem{theorem}{Theorem}
\newtheorem{corollary}[theorem]{Corollary}
\newtheorem{lemma}[theorem]{Lemma}

\newcommand{\eps}{\varepsilon}

\newcommand{\brm}[1]{\operatorname{#1}}
\newcommand{\bb}[1]{\mathbb{#1}}

\DeclareMathOperator*{\argmax}{arg\,max} 
\newcommand{\sgn}{\brm{sgn}}

\title{List coloring with requests}
\author{Zden\v{e}k Dvo\v{r}\'ak\thanks{Charles University, Prague, Czech Republic.
E-mail: {\tt rakdver@iuuk.mff.cuni.cz}. Supported by project 14-19503S (Graph coloring and structure) of Czech Science Foundation.}\\
\and Sergey Norin\thanks{Department of Mathematics and Statistics, McGill University. Email: {\tt snorin@math.mcgill.ca}. Supported by an NSERC grant 418520.}
\\ \and Luke Postle
\thanks{Department of Combinatorics and Optimization,  University of Waterloo. Email: {\tt lpostle@uwaterloo.ca}. Canada Research Chair in Graph Theory. Partially supported by NSERC under Discovery Grant No. 2014-06162, the Ontario Early Researcher Awards program and the Canada Research Chairs program.}
}

\date{}

\begin{document}
\maketitle

\begin{abstract}
Let $G$ be a graph with a list assignment $L$.  Suppose a preferred color is given for some of the
vertices; how many of these preferences can be respected when $L$-coloring $G$? We explore several
natural questions arising in this context, and propose directions for further research.
\end{abstract}

\section{Introduction}

In the precoloring extension problem, one seeks a proper coloring of a graph subject
to some of the vertices having prescribed colors.  This notion appears in
several contexts.  For example, in algorithmic design, a graph being colored may be cut up into
pieces on small vertex cuts, and afterwards one seeks a coloring of the cut vertices
that extends into the pieces~\cite{arnborg1989linear,trfree7,Th2}.  Also, being able to
extend an arbitrary precoloring of a large set of vertices implies existence of
a large number of distinct colorings~\cite{cylgen-part2,lukethe}.

Suppose that the given precoloring does not extend.  Does there at least exist a coloring which matches
the precoloring on many vertices (say, on a constant fraction of the precolored vertices)?  This question is in spirit similar to
various MAX-SAT style constraint satisfaction problems~\cite{cohen2004complete}, seeking the
largest possible number of simultaneously satisfiable constraints of an overdetermined system.
The question is also motivated by a result of Dvo\v{r}\'ak and Sereni~\cite{dser}.
Given a planar graph $G$ and a set $X$ of its edges intersecting
all triangles, is it possible to $3$-color the graph $G-X$ so that a constant fraction of edges of $X$
join vertices of different colors? Dvo\v{r}\'ak and Sereni~\cite{dser} proved that this is equivalent
to a conjecture of Thomassen~\cite{thom-many} that planar triangle-free graphs have exponentially many colorings.

Coming back to our question of satisfying a precoloring on a constant fraction of precolored vertices,
for ordinary proper coloring, the answer is always positive as long
as any coloring of the graph using the fixed number of colors exists. However, the reason why this is the
case is rather unsatisfactory and gives no insight into the chromatic properties of the graph:
suppose that $G$ is a $k$-colorable graph and $r$ is a function assigning preferred colors to some of the
vertices of $G$.  By permuting the colors in a $k$-coloring of $G$, we can easily obtain a $k$-coloring
of $G$ that matches $r$ on at least $|\brm{dom}(r)|/k$ vertices.

This triviality can be avoided by considering list coloring, where the symmetry among the colors disappears.
Let us recall that a \emph{list assignment} for a graph $G$ is a function that to each vertex $v\in V(G)$ assigns
a set $L(v)$ of colors, and an \emph{$L$-coloring} is a proper coloring $\phi$ such that $\phi(v)\in L(v)$ for all $v\in V(G)$.
Let us now introduce the definitions needed to more formally state questions that interest us.
A \emph{request} for a graph $G$ with a list assignment $L$ is a function $r$ with $\brm{dom}(r)\subseteq V(G)$
such that $r(v)\in L(v)$ for all $v\in\brm{dom}(r)$.  For $\eps>0$, a request $r$ is \emph{$\eps$-satisfiable}
if there exists an $L$-coloring $\phi$ of $G$ such that $\phi(v)=r(v)$ for at least $\eps|\brm{dom}(r)|$ vertices $v\in\brm{dom}(r)$.
We say that $G$ with the list assignment $L$ is \emph{$\eps$-flexible} if every request is $\eps$-satisfiable.

The \emph{choosability} of $G$ is the minimum integer $k$ such that $G$ has an $L$-coloring for every assignment $L$ of
lists of size at least $k$.  While the choosability is known to behave rather differently from the chromatic number in general~\cite{alondeg},
for sparse graphs the two parameters are often related.  Recall a graph $G$ is \emph{$d$-degenerate} if every subgraph of $G$ has a vertex
of degree at most $d$, or equivalently, if there exists an ordering of its vertices such that each vertex has at most $d$ neighbors that
precede it in the ordering.  A greedy coloring argument shows that a $d$-degenerate graph has choosability at most $d+1$.
As our first result, we observe that allowing one more color gives flexibility.

\begin{theorem}\label{t:degenerate2}
For every integer $d\ge 0$, there exists $\eps>0$ such that every $d$-degenerate graph with an assignment of lists
of size at least $d+2$ is $\eps$-flexible.
\end{theorem}

We prove Theorem~\ref{t:degenerate2} in a stronger weighted form.
Let $L$ be a list assignment for a graph $G$.  A \emph{weighted request} is a function $w$ that to each pair $(v,c)$ with $c\in L(v)$
assigns a nonnegative real number.  Let $w(G,L)=\sum_{v\in V(G),c\in L(v)} w(v,c)$.
For $\eps>0$, we say that $w$ is \emph{$\eps$-satisfiable} if there exists an $L$-coloring $\phi$ of $G$ such that
$$\sum_{v\in V(G)} w(v,\phi(v))\ge\eps w(G,L).$$
We say that $G$ with the list assignment $L$ is \emph{weighted $\eps$-flexible} if every weighted request is $\eps$-satisfiable.
Of course, weighted $\eps$-flexibility implies $\eps$-flexibility, and thus Theorem~\ref{t:degenerate2} is implied by the following
result.

\begin{theorem}\label{t:degenerate2-weight}
For every integer $d\ge 0$, there exists $\eps>0$ such that every $d$-degenerate graph with an assignment of lists
of size $d+2$ is weighted $\eps$-flexible.
\end{theorem}

To prove weighted $\eps$-flexibility, we use the following observation.
\begin{lemma}\label{lemma:distrib}
Suppose there exists a probability distribution on $L$-colorings $\phi$ of $G$
such that for every $v\in V(G)$ and $c\in L(v)$, $\brm{Prob}[\phi(v)=c]\ge\eps$.
Then $G$ with $L$ is weighted $\eps$-flexible.
\end{lemma}
\begin{proof}
Let $w$ be a weighted request for $G$ and $L$.  Let $\phi$ be chosen at random from the
postulated probability distribution.  By the linearity of expectation,
we have
$$\brm{E}\Biggl[\sum_{v\in V(G)} w(v,\phi(v))\Biggr]=\sum_{v\in V(G),c\in L(v)} \brm{Prob}[\phi(v)=c]\cdot w(v,c)\ge\eps w(G,L),$$
and thus there exists an $L$-coloring $\phi$ with $\sum_{v\in V(G)} w(v,\phi(v))\ge \eps w(G,L)$
as required.
\end{proof}
Let us remark that via linear programming duality, it is also easy to see that weighted $\eps$-flexibility implies
existence of such a distribution.

We do not know whether Theorem~\ref{t:degenerate2} can be strengthened to allow lists of size $d+1$.
The following weaker claim applies e.g. to planar graphs, showing that they are $\eps$-flexible
with assignments of lists of size $6$. Recall that the \emph{maximum average degree} of a graph $G$ is equal to the maximum of $2|E(H)|/|V(H)$ taken over  non-null subgraphs $H$ of $G$. 

\begin{theorem}\label{t:mad}
For every integer $d\ge 2$, there exists $\eps>0$ as follows.
If $G$ is a graph of maximum average degree at most $d$
and choosability at most $d-1$, then $G$ with an assignment of lists of size at least $d$ is $\eps$-flexible.
\end{theorem}
Note that we can only prove the unweighted version of Theorem~\ref{t:mad}.  A weighted result can be proved
under a slightly stronger assumption on the maximum average degree.
\begin{theorem}\label{t:mad2}
For every integer $d\ge 0$, there exists $\eps>0$ as follows.
If $G$ is a graph of maximum average degree less than $d+1+2/(d+4)$
then $G$ with an assignment of lists of size $d+2$ is weighted $\eps$-flexible.
\end{theorem}

A necessary condition for flexibility is that requests with singleton domain can be satisfied.
Coming back to the case of $d$-degenerate graphs with lists of size $d+1$, even proving this
necessary condition is non-trivial and we can only do it in the special case that $d+1$ is a prime.

\begin{theorem}\label{t:degenerate1}
Let $d\ge 2$ be an integer such that $d+1$ is a prime.
If $r$ is a request for a $d$-degenerate graph with an assignment of lists of size at least $d+1$
and $|\brm{dom}(r)|=1$, then $r$ is $1$-satisfiable.
\end{theorem}

It is natural to ask whether $\eps$-flexibility implies weighted $\eps'$-flexibility for some $\eps'\le\eps$
(possibly also depending on the number of colors). This is false, as we show in Section~\ref{s:weighted}.
Theorems~\ref{t:degenerate2-weight} and \ref{t:mad} are proved in Section~\ref{s:degen}.
We prove Theorem~\ref{t:degenerate1} in Section~\ref{s:null}, using Combinatorial Nullstellensatz.

\section{Open problems}

In this paper, we explore only some very basic properties of flexibility, leaving many open questions.
Let us first explicitly state the problem we discussed in the introduction.

\begin{problem}
Does there for every integer $d\ge 0$ exist $\eps>0$ such that every $d$-degenerate graph with an assignment
of lists of size $d+1$ is weighted $\eps$-flexible? Or at least $\eps$-flexible?
\end{problem}
For $d=0$ the problem is trivial. For $d=1$ we ask about forests with assignments of lists of size $2$,
and these are easily seen to be weighted $1/2$-flexible.  Hence, the first open case is $d=2$.

A possibly easier special case of this problem can be formulated in a Brooks'-like setting.
\begin{problem}\label{prob:brooks}
Does there for every integer $\Delta\ge 2$ exist $\eps>0$
such that every connected graph of maximum degree $\Delta$ that is not $\Delta$-regular, with
an assignment of lists of size $\Delta$, is weighted $\eps$-flexible? Or at least $\eps$-flexible?
\end{problem}
Note that it is not sufficient to forbid cliques and odd cycles; for example, there exists an assignment
of lists of size two to vertices of an even cycle such that some request with a singleton domain is not satisfiable.
Conceivably, Problem~\ref{prob:brooks} could have positive answer for $\Delta$-regular graphs
with assignments of lists of size $\Delta$ such that all requests with a singleton domain are satisfiable.

Planar graphs are known to be $5$-choosable~\cite{thomassen1994}, and planar graphs of girth at least $5$
are $3$-choosable~\cite{thomassen1995-34}.  Theorem~\ref{t:mad} (and in the latter case, Theorem~\ref{t:mad2} as well)
show their flexibility with assignments
of lists of size $6$ and $4$, respectively.  We believe these bounds can be improved.

\begin{problem}
Does there exist $\eps>0$ such that every planar graph with an assignment of lists of size $5$ is
weighted $\eps$-flexible? Or at least $\eps$-flexible?
\end{problem}

\begin{problem}
Does there exist $\eps>0$ such that every planar graph of girth at least $5$ with an assignment of lists of size $3$ is
weighted $\eps$-flexible? Or at least $\eps$-flexible?
\end{problem}

However, even simpler questions involving requests seem nontrivial. For example, Euler's formula trivially implies that planar triangle-free 
graphs are $4$-choosable and that planar graphs of girth $6$ are $3$-choosable. Hence we have the following questions.

\begin{problem}
Does there exist $\eps>0$ such that every planar triangle-free graph with an assignment of lists of size $4$ is
weighted $\eps$-flexible? Or at least $\eps$-flexible?
\end{problem}

\begin{problem}
Does there exist $\eps>0$ such that every planar graph of girth at least $6$ with an assignment of lists of size $3$ is
weighted $\eps$-flexible? Or at least $\eps$-flexible?
\end{problem}

Let us remark that Theorem~\ref{t:mad2} implies weighted flexibility of planar graphs of girth at least $12$ with lists of size $3$.

Finally, since the existence of list-colorings in the problems given above is
also known to hold for locally planar graphs (that is, graphs embedded on a
surface with large edge-width), it is natural to ask these questions for
locally planar graphs.

\section{Weighted and unweighted requests}\label{s:weighted}

In this section, we explore the relationship between the weighted and unweighted variants of flexibility.
Let us first mention one triviality: a graph with assignments of lists of size $k$ can only be weighted $\eps$-flexible
for $\eps\le 1/k$, as the request can put equal weights on colors in the list of one vertex.  So, for the weighted
version increasing the size of the lists can actually make things worse, which is not the case for the unweighted version.
However, as we are generally interested in the case where the size of the lists is a fixed constant, this does not bother us.

It is easy to see that if an $n$-vertex graph with a given list assignment is $\eps$-flexible,
it is also weighted $\Omega(1/\log n)$-flexible.

\begin{lemma}\label{l:unwtow}
Let $d\ge 1$ be an integer and $\eps>0$ a real number.
Let $G$ be a graph with $n$ vertices and $L$ an assignment of lists of size $d$ that are $\eps$-flexible.
Then $G$ and $L$ are weighted $\frac{1}{d\log_{1/(1-\eps)} n}$-flexible.
\end{lemma}
\begin{proof}
Let $w$ be a weighted request for $G$ and $L$.
For each $v\in V(G)$, let $c_v\in L(v)$ be a color such that $w(v,c_v)\ge w(v,c)$
for all $c\in L(v)$.  Let $v_1$, \ldots, $v_n$ be an ordering of vertices of $G$ such that
$w(v_1,c_{v_1})\ge w(v_2,c_{v_2})\ge\ldots\ge w(v_n,c_{v_n})$.  
Let $W=\sum_{i=1}^n w(v_i,c_{v_i})$, and note that $W\ge w(G,L)/d$.  

For $1\le k\le n$, let $r_k$ be the request with $\brm{dom}(r_k)=\{v_1,\ldots,v_k\}$ and $r_k(v)=c_v$ for $v\in \brm{dom}(r_k)$.
Since $G$ and $L$ are $\eps$-flexible, there exists an $L$-coloring $\phi_k$ that $\eps$-satisfies $r_k$; i.e., denoting
$M_k=|\{i:1\le i\le k,\phi_k(v_i)=c_{v_i}\}|$, we have $|M_k|\ge \eps k$.

Let $n_0=n$. For $i\ge 1$, let us inductively define $n_i=n_{i-1}-|M_{n_{i-1}}|$, and let $t$ be the smallest index such that $n_t=0$.
Since $|M_{n_{i-1}}|\ge \eps n_{i-1}$, we have $n_i\le (1-\eps)n_{i-1}$, and consequently
$n_i\le (1-\eps)^i n$.  We conclude that $t\le \log_{1/(1-\eps)} n$.  For $0\le i\le t-1$, let $W_i=\sum_{j\in M_{n_i}} w(v_j,c_{v_j})$.
Since the vertices are sorted according to the weights of the requested colors, we have
$$W_i\ge \sum_{j=n_i-|M_{n_i}|+1}^{n_i} w(v_j,c_{v_j})=\sum_{j=n_{i+1}+1}^{n_i} w(v_j,c_{v_j}).$$

Consequently, $\sum_{i=0}^{t-1} W_i\ge W$, and thus for some $i$ with $0\le i\le t-1$, we have $W_i\ge W/t$.
Therefore, letting $\phi=\phi_{n_i}$, we have
$$\sum_{v\in V(G)} w(v,\phi(v))\ge W_i\ge W/t\ge \frac{w(G,L)}{d\log_{1/(1-\eps)} n},$$
showing that $w$ is $\frac{1}{d\log_{1/(1-\eps)} n}$-satisfiable.
\end{proof}

The factor $\Omega(1/\log n)$ in weighted flexibility cannot be improved, as we will show in Corollary~\ref{cor:logfactor}
using a construction described in the following lemma.
Given a graph $G$ with a list assignment $L$ and a set $S=\{v_1,\ldots, v_n\}$ of vertices of $G$ whose list contains color $1$,
we say that a set $R\subseteq\{1,\ldots,n\}$ is \emph{$S$-realizable} if there exists an $L$-coloring $\phi$ of $G$
such that $\{i:1\le i\le n, \phi(v_i)=1\}=R$.

\begin{figure}
\begin{center}
\includegraphics{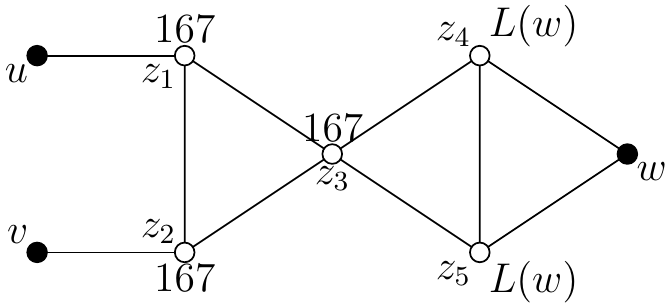}
\end{center}
\caption{A $(u,v)\to w$ gadget.}\label{fig-gadget}
\end{figure}

\begin{lemma}\label{lemma-genweight}
Let $s_1$, \ldots, $s_n$, and $t$ be positive integers.
There exists a graph $G$ with an assignment $L$ of lists of size three
and a subset $S=\{v_1,\ldots, v_n\}$ of vertices of $G$ such that
\begin{itemize}
\item $|V(G)|=O(nt)$,
\item a set $R\subseteq\{1,\ldots,n\}$ is $S$-realizable if and only if $\sum_{i\in R} s_i\le t$, and
\item if $w$ is a weighted request for $G$ and $L$ such that $w(v,1)=0$ for all $v\in S$,
then $w$ is $1/4$-satisfiable.
\end{itemize}
\end{lemma}
\begin{proof}
To construct the graph $G$, we need an auxiliary construction.  Suppose we are given some graph $G'$
and vertices $u,v,w\in V(G')$ whose lists are either ${1,2,3}$ or ${1,4,5}$, with possibly $u$ and $v$ denoting
the same vertex.
By \emph{adding a $(u,v)\to w$ gadget} to $G'$ we mean adding vertices $z_1$, \ldots, $z_5$,
edges of the triangles $z_1z_2z_3$ and $z_3z_4z_5$, and edges $uz_1$, $vz_2$,
$wz_4$, and $wz_5$, with lists $L(z_1)=L(z_2)=L(z_3)=\{1,6,7\}$ and $L(z_4)=L(z_5)=L(w)$;
see Figure~\ref{fig-gadget}.
Observe that an $L$-coloring of $u$, $v$, and $w$ extends to an $L$-coloring of the
described subgraph induced by $\{u,v,w,z_1,\ldots,z_5\}$ if and only if either at least one
of $u$ and $v$ has color different from $1$, or $w$ has color $1$; i.e., $u$ and $v$ both
being colored by $1$ implies that $w$ is colored by $1$.

\begin{figure}
\begin{center}
\includegraphics{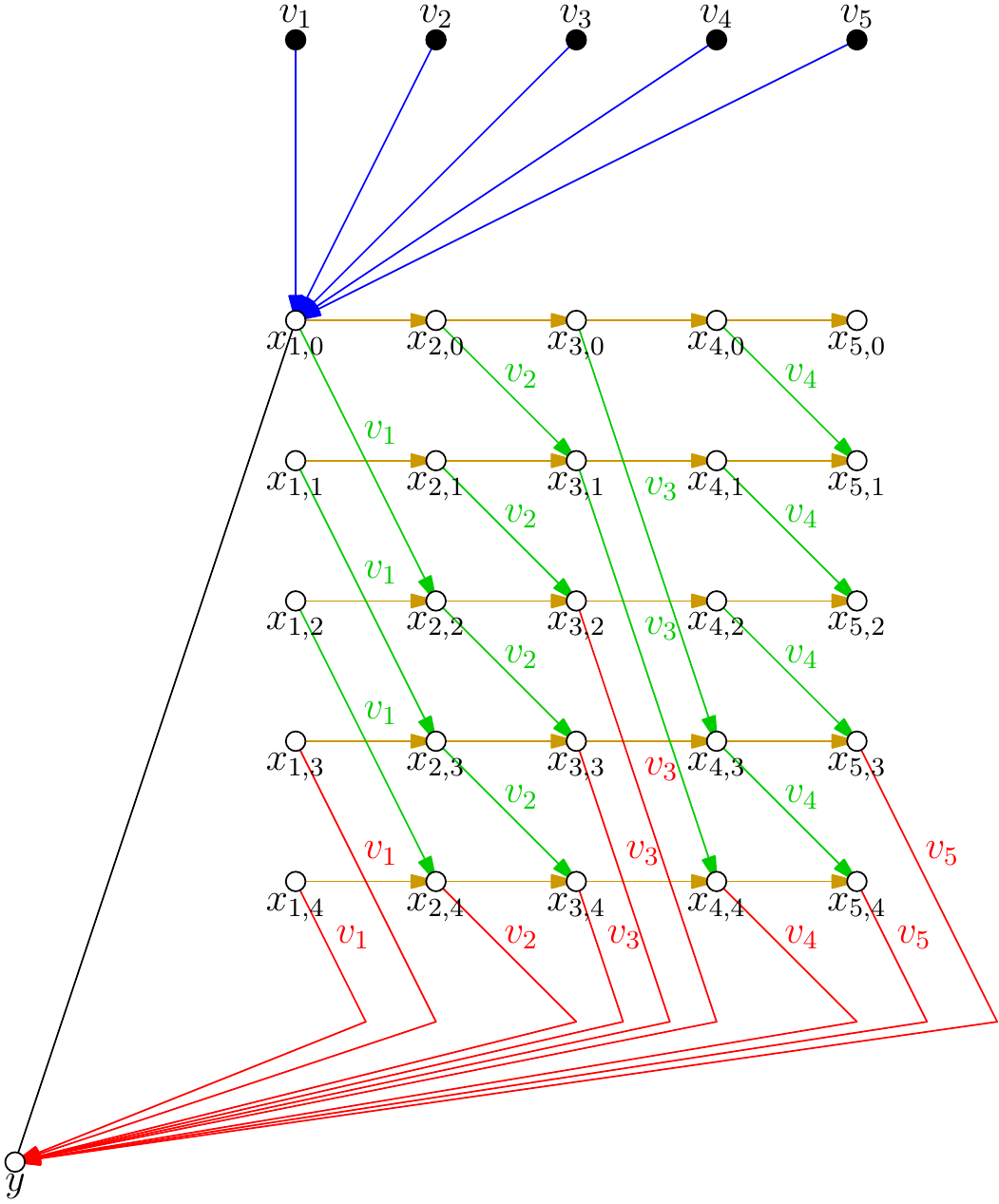}
\end{center}
\caption{Construction from Lemma~\ref{lemma-genweight}, $(s_1,\ldots,s_n)=(2,1,3,1,2)$, $t=4$.
Edges in blue, yellow, green, and red depict gadgets added in steps (a), (b), (c), and (d), respectively,
directed from the $v$-vertex to the $w$-vertex of the gadget.  When the $u$-vertex of the gadget is different from the $v$-vertex,
it is given by the label on the edge.}\label{fig-machine}
\end{figure}

The construction of $G$ (illustrated in Figure~\ref{fig-machine}) starts with vertices of $S$ with lists $\{1,2,3\}$,
and vertices $x_{i,j}$ for $1\le i\le n$ and $0\le j\le t$ and a vertex $y$ with lists $\{1,4,5\}$,
and an edge $yx_{1,0}$.  Next,
\begin{itemize}
\item[(a)] for $1\le i\le n$, add a $(v_i,v_i)\to x_{1,0}$ gadget,
\item[(b)] for $1\le i\le n-1$ and $0\le j\le t$, add an $(x_{i,j},x_{i,j})\to x_{i+1,j}$ gadget,
\item[(c)] for $1\le i\le n-1$ and $0\le j\le t-s_i$, add a $(v_i,x_{i,j})\to x_{i+1,j+s_i}$ gadget, and
\item[(d)] for $1\le i\le n$ and $t-s_i+1\le j\le t$, add a $(v_i,x_{i,j})\to y$ gadget.
\end{itemize}
Note that $|V(G)|<11n(t+2)$.

Consider now the $L$-colorings $\phi_1$, \ldots, $\phi_4$ of $G$ in which vertices get colors
according to the following table (the last two columns display the colors assigned to the vertices
of the gadgets added in steps (a), (b), (c), or (d), respectively).

\vspace{-3mm}

\begin{center}
\begin{tabular}{|c|c|c|c|c|c|}
\hline
Coloring& $S$ & $x_{i,j}$ & $y$ & (a), (b), (c) & (d) \\
& & $(1\le i\le n, 0\le j\le t)$ & & $z_1,z_2,z_3,z_4,z_5$ & $z_1,z_2,z_3,z_4,z_5$ \\
\hline
$\phi_1$ & $2$ & $1$ & $4$ & $6,7,1,4,5$ & $1,6,7,1,5$ \\
$\phi_2$ & $3$ & $1$ & $5$ & $7,6,1,5,4$ & $1,7,6,4,1$ \\
$\phi_3$ & $3$ & $4$ & $1$ & $1,6,7,1,5$ & $6,1,7,5,4$ \\
$\phi_4$ & $3$ & $5$ & $1$ & $7,1,6,4,1$ & $7,6,1,4,5$ \\
\hline
\end{tabular}
\end{center}

\smallskip

Let $\phi$ be chosen as one of $\phi_1$, \ldots, $\phi_4$ uniformly at random.  Then for all $v\in V(G)$ and $c\in L(v)$
such that either $v\not\in S$ or $c\neq 1$ we have $\brm{Prob}[\phi(v)=c]\ge 1/4$, implying as in Lemma~\ref{lemma:distrib}
that $\phi$ with non-zero probability
$1/4$-satisfies any weighted request such that $w(v,1)=0$ for all $v\in S$.

Consider now any $R\subseteq\{1,\ldots,n\}$; when does an $L$-coloring $\phi$ of $G$ such that $\{i:1\le i\le n, \phi(v_i)=1\}=R$ exist?
If $R=\emptyset$, then this is the case due to say the $L$-coloring $\phi_1$; hence, suppose that $R\neq\emptyset$.
In that case, the gadgets (a) force $\phi(x_{1,0})=1$, and because of the edge $x_{1,0}y$, we may without loss of generality assume $\phi(y)=4$.
Note that since colors $2$ and $3$ only appear in the lists of vertices of $S$, we can without loss of generality assume that $\phi(v_i)=2$ for $i\in \{1,\ldots,n\}\setminus R$.
Gadgets (b) and (c) force that for $1\le i\le n$, if there exists a set $R'\subseteq R$ with $\max(R')<i$ and $t'\colonequals \sum_{j\in R'} s_j\le t$, then
$\phi(x_{i,t'})=1$.  If $\sum_{j\in R} s_j>t$, the gadgets (d) prevent the existence of $\phi$, since $\phi(y)\neq 1$.
Otherwise, we can set $\phi(x_{i,j})=4$ for all $i$ and $j$ whose color is not forced by a set $R'\subseteq R$ as described before,
and extend $\phi$ to an $L$-coloring of $G$ by choosing the colorings in the gadgets arbitrarily.
\end{proof}

\begin{corollary}\label{cor:logfactor}
For every positive integer $k$, there exists a graph $G$ with $O(4^k)$ vertices and an assignment $L$ of lists of size three
such that they are $1/6$-flexible and not weighted $\eps$-flexible for any $\eps>1/k$.
\end{corollary}
\begin{proof}
Let $G$, $S=\{v_1,\ldots,v_n\}$, and $L$ be constructed using Lemma~\ref{lemma-genweight} with $n=2^k-1$, $s_i$ with $1\le i\le n$
equal to the largest power of two such that $is_i\le 2^k-1$ (i.e., $s_1=2^{k-1}$, $s_2=s_3=2^{k-2}$, \ldots, $s_{2^{k-1}}=\ldots=s_n=1$),
and $t=2^{k-1}$.  Observe that if $1\le i_1\le i_2\le n$ satisfy $i_2\le 2i_1-1$, then $\sum_{j=i_1}^{i_2} s_j\le 2^{k-1}$.

Consider any request $r$.  Let $r_1$ and $r_2$ be the restrictions of $r$ to
$\brm{dom}(r_1)=\{v\in S\cap \brm{dom}(r):r(v)=1\}$ and $\brm{dom}(r_2)=\brm{dom}(r)\setminus \brm{dom}(r_1)$.
According to Lemma~\ref{lemma-genweight}, the request $r_2$ is $1/4$-satisfiable, and thus there exists an $L$-coloring
$\phi_2$ of $G$ such that $|\{v\in\brm{dom}(r_2):\phi_2(v)=r(v)\}|\ge |\brm{dom}(r_2)|/4$.
Furthermore, let $R_1=\{i:1\le i\le n, v_i\in \brm{dom}(r_1)\}$
and let $R$ consist of $\lceil |R_1|/2\rceil$ largest elements of $R_1$.  Since the sequence $s_1$, $s_2$, \ldots, $s_n$
is non-increasing, we have
$$\sum_{i\in R} s_i\le \sum_{i=\lfloor |R_1|/2\rfloor+1}^{|R_1|} s_i\le 2^{k-1},$$
and the $L$-coloring $\phi_1$ such that $\phi_1(v_i)=1$ for $i\in R_1$ that exists
according to Lemma~\ref{lemma-genweight} shows that $r_1$ is $1/2$-satisfiable.
Considering $i\in\{1,2\}$ with larger $|\{v\in\brm{dom}(r):\phi_i(v)=r(v)\}|$,
we conclude that $r$ is $1/6$-satisfiable, and consequently $G$ and $L$ are $1/6$-flexible.

Let us now consider the weighted request $w$ such that $w(v_i,1)=s_i$ for $1\le i\le n$ and
$w(v,c)=0$ if $v\not\in S$ or $c\neq 1$.  We have $w(G,L)=k2^{k-1}$,
and according to Lemma~\ref{lemma-genweight}, every $L$-coloring $\phi$ of $G$ satisfies
$$\sum_{1\le i\le n, \phi(v_i)=1} s_i\le 2^{k-1},$$
showing that $w$ is not $\eps$-satisfiable for any $\eps>1/k$.
\end{proof}

Also, we have another interesting corollary: consider the graph $G$ and list assignment $L$ of
Lemma~\ref{lemma-genweight} with $s_1=\ldots=s_n=1$ and $t=1$.  Then any request with singleton domain
is $1$-satisfiable, but $G$ with $L$ are not $\eps$-flexible for any $\eps>1/n$.

\section{Flexibility of degenerate graphs}\label{s:degen}

We will need the following simple result on list coloring.
\begin{lemma}\label{lemma:degcol}
Let $G$ be a connected graph and let $L$ be a list assignment for $G$ such that $|L(v)|>\deg(v)$ for all $v\in V(G)$.
Then for every $v\in V(G)$ and a color $c$, there exists an $L$-coloring of $G$ such that no vertex other than $v$ is
assigned the color $c$.
\end{lemma}
\begin{proof}
Let $L'$ be the list assignment for $G$ such that $L'(w)=L(w)\setminus\{c\}$ for $w\in V(G)\setminus \{v\}$
and $L'(v)=L(v)$.  Let $T$ be a spanning tree of $G$ rooted in $v$.  List the vertices of $G$ in the reverse
DFS order of $T$ and color them greedily from the list assignment $L'$; this is possible since when coloring
a vertex $w\neq v$, at least one neighbor of $w$ (its parent in $T$) is still uncolored, and thus the number of colors in $L'(w)$
(which is at least $\deg(w)$ by the assumptions) is greater than the number of already colored neighbors of $w$.
\end{proof}

We say that a graph $G$ is \emph{weakly $d$-degenerate} if for every subgraph $G'$ of $G$,
either the minimum degree of $G'$ is at most $d$, or $G'$ contains a set of $d+1$ vertices of
degree $d+1$ inducing a connected subgraph.
Both Theorems~\ref{t:degenerate2-weight} and \ref{t:mad2} are consequences of the following claim.

\begin{lemma}\label{lemma:flex-wdeg}
For every integer $d\ge 0$, there exist $\eps,\delta>0$ as follows.
Let $G$ be a graph and let $L$ be an assignment of lists of size $d+2$ to vertices of $G$.
If $G$ is weakly $d$-degenerate, then there exists a probability distribution on $L$-colorings $\phi$ of $G$ such that
\begin{itemize}
\item[\textrm{(i)}] for all $v\in V(G)$ and $c\in L(v)$ we have $\brm{Prob}[\phi(v)=c]\ge\eps$, and
\item[\textrm{(ii)}] for any set $S\subseteq V(G)$ of size at most $d$ and for any color $c$,
$$\brm{Prob}[(\forall v\in S)\,\phi(v)\neq c]\ge\delta^{|S|}.$$
\end{itemize}
\end{lemma}
\begin{proof}
Let $\delta=1/(d+2)^{d+1}$ and $\eps=\delta^{d+1}$.  We prove the lemma by induction on $|V(G)|$.
If $G$ contains a vertex $w$ of degree at most $d$, then let $P=\{w\}$.  Otherwise, since $G$
is weakly $d$-degenerate, there exists a set $P$ of $d+1$ vertices of degree $d+1$ inducing a connected
subgraph of $G$.  A random $L$-coloring $\phi$ of $G$ is chosen as follows: We choose an $L$-coloring
$\phi_0$ of $G-P$ at random from the probability distribution obtained by the induction hypothesis.
Let $L'$ be the list assignment for $G[P]$ such that
$$L'(v)=L(v)\setminus\{\phi_0(u):uv\in E(G), u\not\in P\}$$
for all $v\in P$.  We choose an $L'$-coloring $\phi_1$ uniformly at random among all $L'$-colorings of $G[P]$,
and we let $\phi$ be the union of the colorings $\phi_0$ and $\phi_1$.

First, let us show that condition (ii) holds.  Let $S_1=S\setminus P$ and $S_2=S\cap P$.
By the induction hypothesis for $G-P$, we have $\brm{Prob}[(\forall v\in S_1)\,\phi(v)\neq c]=\brm{Prob}[(\forall v\in S_1)\,\phi_0(v)\neq c]\ge\delta^{|S_1|}$.
If $S_1=S$, this implies (ii).  Hence, suppose that $|S_1|\le |S|-1$.
Let us fix $\phi_0$, and consider the probability that $\phi_1$ gives all vertices of $S_2$ color different from $c$.
If $P$ consists of a single vertex $w$ of degree at most $d$, then $|L'(w)|\ge 2$, and thus $\phi_1(w)\neq c$ with probability
at least $1/2$.  Otherwise, $P$ consists of $d+1$ vertices of degree $d+1$.
Observe that $|L'(v)|>\deg_{G[P]} v$ for all $v\in P$. Since $G[P]$ is connected and $|P|=d+1>|S_2|$, Lemma~\ref{lemma:degcol}
implies that there exists an $L'$-coloring of $G[P]$ in that no vertex of $S_2$ is assigned color $c$.
Since $\phi_1$ is chosen uniformly among the at most $(d+2)^{d+1}$ $L'$-colorings of $G[P]$, the
probability that no vertex of $S_2$ is assigned color $c$ by $\phi_1$ is at least $1/(d+2)^{d+1}=\delta$.

Consequently, conditionally under the assumption that $\phi_0$ does not assign color $c$ to
any vertex of $S_1$, the probability that $\phi_1$ does not assign color $c$ to any vertex of $S_2$ is at least $\delta$.
Hence,
$$\brm{Prob}[(\forall v\in S)\phi(v)\neq c]\ge \brm{Prob}[(\forall v\in S_1)\phi_0(v)\neq c]\cdot\delta\ge \delta^{|S_1|+1}\ge \delta^{|S|}$$
as required.

Let us now prove that (i) holds.  Consider any vertex $v\in V(G)$ and a color $c\in L(v)$.  If $v\not\in P$,
then $\brm{Prob}[\phi(v)=c]=\brm{Prob}[\phi_0(v)=c]\ge \eps$ by the induction hypothesis for $G-P$.
Hence, assume that $v\in P$. Let $S$ be the set of neighbors
of $v$ in $V(G)\setminus P$; we have $|S|\le d$, and thus by (ii), with probability
at least $\delta^d$ no vertex of $S$ is assigned color $c$ by $\phi_0$.  If that is the case, we have $c\in L'(v)$.
Trivially if $|P|=1$ and by Lemma~\ref{lemma:degcol} if $|P|=d+1$, there exists an $L'$-coloring
of $G[P]$ with no vertex other than $v$ colored by $c$.  We can recolor $v$ with $c$ if needed, showing that there
exists an $L'$-coloring of $G[P]$ in that $v$ is assigned color $c$.  Since $\phi_1$ is chosen uniformly among the
$L'$-colorings of $G[P]$, it follows that conditionally under the assumption that no vertex of $S$ is colored by $c$,
the probability that $\phi(v)=c$ is at least $\delta$.  Therefore,
$\brm{Prob}[\phi(v)=c]\ge \delta^d\cdot\delta=\eps$.
\end{proof}

Since every $d$-degenerate graph is also weakly $d$-degenerate, Lemmas~\ref{lemma:distrib} and \ref{lemma:flex-wdeg}(i)
imply Theorem~\ref{t:degenerate2-weight}.
Similarly, to prove Theorem~\ref{t:mad2}, it suffices show weak $d$-degeneracy of graphs of maximum average degree
less than $d+1+2/(d+4)$, which is a consequence of the following lemma.

\begin{lemma}\label{lemma:sg}
Let $d\ge 0$ be an integer.  If a graph $G$ has average degree less than $d+1+2/(d+4)$,
then either $G$ contains a vertex of degree at most $d$, or a vertex of degree $d+1$
with at most one neighbor of degree greater than $d+1$.
\end{lemma}
\begin{proof}
Suppose for a contradiction that $G$ has minimum degree at least $d+1$ and that
each vertex of $G$ of degree $d+1$ has at least two neighbors of larger degree.
Let us assign charge $\deg(v)-(d+1+2/(d+4))$ to each vertex $v\in V(G)$; since the average degree
of $G$ is less than $d+1+2/(d+4)$, the sum of the charges is negative.  Now, each vertex of
degree at least $d+2$ sends $1/(d+4)$ of its charge to each adjacent vertex.
After this redistribution, the charge of each vertex $v$ of degree at least $d+2$
is at least
$$\deg(v)-\Bigl(d+1+\frac{2}{d+4}\Bigr)-\frac{\deg(v)}{d+4}=\frac{(d+3)(\deg(v)-d-2)}{d+4}\ge 0.$$
Each vertex $v$ of degree $d+1$ starts with charge $-2/(d+4)$ and receives charge $1/(d+4)$
from at least two of its neighbors, and thus the final charge of $v$ is non-negative.
Hence, the sum of the charges is non-negative, which is a contradiction since the
redistribution did not change the total amount of charge.
\end{proof}

To prove Theorem~\ref{t:mad}, let us first consider its special case where the requests form an independent set.

\begin{lemma}\label{l:madind}
Let $d\ge 2$ be an integer, let $G$ be a graph of maximum average degree at most $d$ such that $G$ is $(d-1)$-choosable,
let $L$ be an assignment of lists of size at least $d$ for $G$, and let $r$ be a request.
If $\brm{dom}(r)$ is an independent set in $G$, then $r$ is $\frac{1}{2d^{2d}}$-satisfiable.
\end{lemma}
\begin{proof}
Let $\eps=\frac{1}{2d^{2d}}$.

We prove the claim by induction on the number of vertices of $G$.  If there exists a vertex $v\in V(G)\setminus\brm{dom}(r)$
of degree less than $d$, then by the induction hypothesis, there exists an $L$-coloring of $G-v$ that
$\eps$-satisfies $r$, and this coloring can be extended to $G$ by giving $v$ a color in $L(v)$ not appearing on its neighbors.

Hence, we can assume that all vertices of $V(G)\setminus\brm{dom}(r)$ have degree at least $d$.
Since $G$ has average degree at most $d$, the vertices in $\brm{dom}(r)$ have average degree at most $d$.
Consequently, less than half of the vertices of $\brm{dom}(r)$ has degree greater than $2d$.

Without loss of generality, we can assume that all lists assigned by $L$ have size exactly $d$.
For each vertex $v\in V(G)\setminus\brm{dom}(r)$, choose a color $c_v\in L(v)$ independently uniformly at random,
and let $L'(v)=L(v)\setminus \{c_v\}$. For $v\in \brm{dom}(r)$, let $L'(v)=L(v)$.  Since $G$ is $(d-1)$-choosable,
$G$ has an $L'$-coloring $\phi_0$.  Let $\phi$ be obtained from $\phi_0$ by, for each $v\in \brm{dom}(r)$ such that the color
$r(v)$ does not appear on the neighbors of $v$ in the coloring $\phi_0$, changing the color of $v$ to $r(v)$.

Consider a vertex $v\in \brm{dom}(r)$ of degree at most $2d$.  For each neighbor $u$ of $v$, the probability that $L'(u)$
contains $r(v)$ is at most $1-1/d$.  Hence, the probability that $\phi(v)=r(v)$ is at least $\frac{1}{d^{\deg(v)}}\ge\frac{1}{d^{2d}}=2\eps$.
Since at least half of the vertices of $\brm{dom}(r)$ have degree at most $2d$, we have
$$\brm{E}[|\{v\in \brm{dom}(r):\phi(v)=r(v)|]\ge\eps|\brm{dom}(r)|,$$
and thus with non-zero probability the coloring $\phi$ $\eps$-satisfies the request~$r$.
Consequently, $r$ is $\eps$-satisfiable.
\end{proof}

\begin{proof}[Proof of Theorem~\ref{t:mad}]
Let $\eps=\frac{1}{2(d-1)d^{2d}}$. Let $L$ be an assignment of lists of size at least $d$ to $G$ and let $r$ be a request.

Since $G$ has choosability at most $d-1$, it is $(d-1)$-colorable.  Considering the color class
with largest intersection with $\brm{dom}(r)$, it follows that $G$ contains an independent set $A$
such that $|A\cap \brm{dom}(r)|\ge |\brm{dom}(r)|/(d-1)$.  Let $r_1$ be the restriction of $r$
to $A$.  By Lemma~\ref{l:madind}, $r_1$ is $(d-1)\eps$-satisfiable, and thus $r$ is $\eps$-satisfiable.
\end{proof}

\section{Satisfying one request on degenerate graphs}\label{s:null}

In this Section we prove Theorem~\ref{t:degenerate1}. 
Let $G$ be a graph with a fixed ordering of its vertices $(v_1,v_2,\ldots,v_n)$. The \emph{graph polynomial of $G$} is defined as 
$$p_G(x_1,x_2,\ldots,x_n)=\prod_{v_iv_j \in E(G), \: i< j}(x_j -x_i).$$ 
Our main tool is a special case of Alon's Combinatorial Nullstellensatz. 

\begin{theorem}[Alon, Tarsi~\cite{alon1992colorings}]\label{t:nullstellensatz}
Let $G$ be a graph with vertex set $(v_1,v_2,\ldots,v_n)$. Suppose that the coefficient of $x_1^{d_1}\ldots x_n^{d_n}$ in $p_G(x_1,x_2,\ldots,x_n)$ is non-zero. Let $L$ be a list assignment for $G$ such that $|L(v_i)| \geq d_i+1$ for $i=1,2,\ldots, n$. Then $G$ is $L$-colorable.
\end{theorem}

While in most applications, including the original application
in~\cite{alon1992colorings}, the coefficient in Theorem~\ref{t:nullstellensatz}
is interpreted in terms of orientations of $G$, we find it equally convenient
to work directly with the graph polynomial.

Let us start by some preparatory work.  For an integer $d$, let $S_d$ denote the set of all permutations
of the set $\{1,\ldots,d\}$. Let $S^0_d$ denote the set of all bijections from $\{1,\ldots,d\}$ to $\{0,\ldots,d-1\}$.
\begin{lemma}\label{lemma-claim2}
Let $d$ be a positive integer and let $r:\{1,\ldots,d\}\to\mathbb{Z}_0^0$ be a function such that
$\sum_{i=1}^d r(i)=d$.  For $\pi \in S^0_d$, $\pi+r\in S_d$ if and only if
\begin{equation}\label{eq-claim2}
\pi(t)=\sum_{j:\pi(j)<\pi(t)} r(j)
\end{equation}
for every $t\in \{1,\ldots, d\}$ such that $r(t)>0$.
\end{lemma}
\begin{proof}
Let $\sigma=\pi+r$.

Suppose first that $\sigma\in S_d$ and consider $t\in \{1,\ldots, d\}$ such that $r(t)>0$.
Let $k=\pi(t)$ and $K=\pi^{-1}(\{0,1,\ldots,k-1\})$.  If $i\in \{1,\ldots,d\}\setminus K$,
then $\pi(i)\ge k$, and thus $\sigma(i)>k$ (when $\pi(i)=k$, we have $i=t$ and $\sigma(i)>\pi(i)=k$ since $r(t)>0$).
Consequently, $\sigma^{-1}(\{1,\ldots,k\})\subseteq K$, and since $|K|=k$, we have $\sigma^{-1}(\{1,\ldots,k\})=K$.
It follows that
\begin{align*}
\sum_{i\in K} r(i)&=\sum_{i\in K} (\sigma(i)-\pi(i))=\sum_{i\in K} \sigma(i)-\sum_{i\in K} \pi(i)\\
&=\sum_{s=1}^k s-\sum_{s=0}^{k-1} s=k=\pi(t),
\end{align*}
which is equivalent to (\ref{eq-claim2}).

Conversely, let $R=\{t\in \{1,\ldots, d\}:r(t)>0\}$ and suppose that (\ref{eq-claim2}) holds for every $t\in R$.
In particular, for $t\in R$ we have
$$\sigma(t)=\pi(t)+r(t)=\sum_{j\in R:\pi(j)\le \pi(t)} r(j),$$
and for $t=\argmax_{j \in R}\pi(j)$
$$\sigma(t)=\sum_{j\in R} r(j)=d$$
holds by the assumptions.
Consider any $s\in\{1,\ldots, d-1\}$ and let $t=\pi^{-1}(s)$.  If $t\not\in R$, then $\sigma(t)=\pi(t)=s$.
If $t\in R$, then consider $t'=\argmax_{j \in R, \pi(j)<s} \pi(j)$; note that some $j\in R$ such that
$\pi(j)<s$ exists, since $$1\le s=\pi(t)=\sum_{j \in R, \pi(j)<s} r(j)$$ by (\ref{eq-claim2}).
We have
$$\sigma(t')=\sum_{j\in R:\pi(j)\le \pi(t')} r(j)=\sum_{j\in R:\pi(j)<\pi(t)} r(j)=\pi(t)=s.$$
Consequently, $\{1,\ldots,d\}\subseteq \sigma(\{1,\ldots,d\})$, and thus $\sigma$ is a permutation
of $\{1,\ldots,d\}$.
\end{proof}

The \emph{sign} of a bijection $\pi\in S_d\cup S^0_d$ is defined as
$$\sgn(\pi)=(-1)^{|\{(i,j):1\le i<j\le d,\, \pi(i)>\pi(j)\}|}.$$

\begin{corollary}\label{cor-claim2-count}
Let $d$ be a positive integer and let $r:\{1,\ldots,d\}\to\mathbb{Z}_0^0$ be a function such that
$\sum_{i=1}^d r(i)=d$.  Let $R=\{t\in\{1,\ldots,d\}:r(t)>0\}$ and let $k=|R|$.
There are exactly $k!(d-k)!$ bijections $\pi \in S^0_d$ such that $\pi + r \in S_d$.
Furthermore,
$$\sgn(\pi)\sgn(\pi+r)=(-1)^{k+d}$$
for each such bijection $\pi$.
\end{corollary}
\begin{proof}
If $\omega:\{1,\ldots,k\}\to R$ is a bijection and $\pi\in S^0_d$ satisfies
\begin{equation}\label{eq-count}
\pi(\omega(i))=\sum_{j=1}^{i-1} r(\omega(j))
\end{equation}
for $i\in\{1,\ldots,k\}$, then $\pi+r\in S_d$ by Lemma~\ref{lemma-claim2} and $\pi(\omega(1))<\pi(\omega(2))<\ldots<\pi(\omega(k))$.
Conversely, for any bijection $\pi\in S^0_d$, there exists a unique bijection $\omega:\{1,\ldots,k\}\to R$
such that $\pi(\omega(1))<\pi(\omega(2))<\ldots<\pi(\omega(k))$, and if $\pi+r\in S_d$, then (\ref{eq-count}) holds by Lemma~\ref{lemma-claim2}.

We conclude that each bijection $\pi \in S^0_d$ such that $\pi + r \in S_d$ can be obtained by first choosing
a bijection $\omega:\{1,\ldots,k\}\to R$ (in one of $k!$ ways), fixing the values of $\pi$ on $R$ according to (\ref{eq-count}),
and choosing the rest of values arbitrarily (in $(d-k)!$ ways).  Hence, the number of such bijections is exactly $k!(d-k)!$.

Consider $\sigma=\pi+r$.  For $t\not\in R$ we have $\sigma(t)=\pi(t)$.  For $1\le i\le k-1$, we have
$$\sigma(\omega(i))=r(\omega(i))+\sum_{j=1}^{i-1} r(\omega(j))=\pi(\omega(i+1))$$
by (\ref{eq-count}),
and
$$\sigma(\omega(k))=\sum_{j=1}^k r(\omega(i))=d.$$  Furthermore, note that $\pi(\omega(1))=0$.
Hence, $\sigma$ is obtained from $\pi$ by first replacing $0$ by $d$ (multiplying the sign by $(-1)^{d-1}$),
then shifting the values cyclically on $R$ 
in order given by $\omega$ (multiplying the sign by $(-1)^{k-1}$).
Consequently, $\sgn(\pi)\sgn(\pi+r)=(-1)^{k+d}$.
\end{proof}

We are now ready to apply Theorem~\ref{t:nullstellensatz} to prove the following technical generalization of Theorem~\ref{t:degenerate1}.

\begin{theorem}\label{t:degenerate1a} Let $d\ge 2$ be an integer such that $d+1$ is  prime.
Let $G$ be a $d$-degenerate graph with vertex set $\{v_1,v_2,\ldots,v_n\}$. Let
$r:\{1,\ldots,n\}\to\bb{Z}^+_0$ satisfy
$\sum_{i=1}^nr(i)=d$. Let $L$ be a list assignment for $G$ such that $|L(v_i)|
\geq d+1-r(i)$ for $1 \leq i \leq n$. Then $G$ is $L$-colorable. 	
\end{theorem}

\begin{proof}
Without loss of generality we assume that $G$ is a maximal $d$-degenerate graph, $n=|V(G)| \geq d$, and that the vertices $(v_1,v_2,\ldots,v_n)$ are
ordered so that $\{v_1,v_2,\ldots,v_d\}$ is a clique in $G$, and  the vertex $v_i$ has exactly $d$ neighbors in the set $\{v_j \: | \: j < i\}$ for
every $i > d$. Let $p_G(x_1,x_2,\ldots,x_n)$ be the graph polynomial of $G$.
Let $h=\prod_{i=1}^{n}x_i^{r(i)}$, and note that $h$ is a monic monomial in variables $x_1,x_2,\ldots,x_n$ of degree $d$.
Let $\sigma \in S_d$ be a permutation. We denote by $c_G(h,\sigma)$
the coefficient of the term $$\frac{1}{h}\prod_{i=1}^{d}x_i^{\sigma(i)}\prod_{i=d+1}^n x_i^{d}$$
in $p_G$. (Note that it is possible that the above product contains negative powers of some of the variables, in which case the coefficient is naturally zero.)
The degree of $x_i$ in the term above does not exceed $d-r(i)$ for every $1 \leq i \leq n$. Thus by Theorem~\ref{t:nullstellensatz}  if for every monomial $h$ as above there exists $\sigma \in S_d$ such that $c_G(h,\sigma) \neq 0$, then Theorem~\ref{t:degenerate1a} holds.

Let $$c_G(h)=\sum_{\sigma \in S_d} \sgn(\sigma)c_G(h,\sigma).$$
By the observation above, 
the next claim implies the theorem.

\vskip 5pt

\noindent
{\bf Claim 1:} $c_G(h) \equiv -1 \pmod{d+1}$.

\vskip 5pt
	
\noindent It remains to establish the claim. The proof is by induction on $n=|V(G)|$. 

Let us start with the base case $n=d$.  In this case $p_G(x_1,\ldots,x_d)$ is the Vandermonde polynomial, and thus
$$p_G(x_1,x_2,\ldots,x_d)=\sum_{\pi \in S^0_d}\sgn(\pi) x_1^{\pi(1)}x_2^{\pi(2)}\ldots x_d^{\pi(d)}.$$
Consequently,
$$
c_G(h,\sigma)=\begin{cases}
\sgn(\sigma-r)&\text{if $\sigma-r\in S^0_d$}\\
0&\text{otherwise}
\end{cases}
$$
and
\begin{equation}\label{e:base}
c_G(h) = \sum_{\pi \in S^0_d, \pi+r \in S_d} \sgn(\pi) \sgn(\pi+r).
\end{equation}
Since $d+1\ge 3$ is prime, $d$ is even.  Letting $k=|\{t\in\{1,\ldots,d\}:r(t)>0\}|$,
Corollary~\ref{cor-claim2-count} implies
\begin{equation}\label{e:binom}
c_G(h)=(-1)^k k! (d-k)! \equiv d(d-1)\ldots(d-k+1)(d-k)! = d! \equiv -1,
\end{equation}
where the congruences here and in the sequel are modulo $d+1$ and the last congruence is Wilson's theorem.
This finishes the proof of the base case of Claim 1.

For the induction step, let $H$ be the graph obtained from $G$ by deleting $v_n$, and let $v_{i_1},v_{i_2},\ldots,v_{i_d}$ be the neighbors of $v_n$. Then \begin{equation}\label{e:step}
p_G=(x_n - x_{i_1})(x_n - x_{i_2})\ldots(x_n-x_{i_d})p_H.
\end{equation}
For $S \subseteq \{1,2,\ldots,d\}$ of size $r(n)$ define $h_S = h x_n^{-r(n)}\prod_{j \in S} x_{i_j}$. From (\ref{e:step}) we have  
\begin{equation}\label{e:step2}
c_G(h,\sigma) = (-1)^{r(n)}\sum_{S \subseteq \{1,2,\ldots,d\}, \: |S|=r(n)}c_H(h_S,\sigma).
\end{equation}
By the induction hypothesis each summand in (\ref{e:step2}) is congruent to $-1$ modulo $d+1$, and we have
$$c_G(h,\sigma) \equiv (-1)^{r(n)+1}\binom{d}{r(n)} \equiv -1,$$
as $\binom{d}{k} \equiv  (-1)^{k}\pmod{d+1}$ for all $k$, as shown in (\ref{e:binom}).
 \end{proof} 

\section*{Acknowledgments}

The research leading to this paper was started at ``New Trends in Graph Coloring'' workshop held at Banff International Research Station.


\end{document}